\documentclass[a4paper,11pt,reqno]{amsart}
\usepackage[margin=2cm]{geometry}
\usepackage[main=english, russian]{babel}
\usepackage[utf8]{inputenc}
\usepackage{amssymb}
\usepackage{amsthm}
\usepackage{amsmath}
\usepackage{esvect}
\usepackage{hyperref}
\usepackage{cleveref}
\usepackage{thmtools}
\usepackage{thm-restate}
\usepackage{setspace}
\usepackage{verbatim}
\usepackage[dvips]{graphicx}
\usepackage{t1enc}
\usepackage{color}
\usepackage{pgf,tikz}
\usepackage{mathrsfs}
\usetikzlibrary{arrows, automata, positioning}

\usepackage{adjustbox,tikz,calc,graphics,babel,standalone}

\usepackage{enumerate, enumitem}

\usetikzlibrary{shapes.misc,calc,intersections,patterns,decorations.pathreplacing}
\usetikzlibrary{arrows,shapes,positioning}
\usetikzlibrary{decorations.markings}
\usetikzlibrary{graphs,graphs.standard,chains,fit}
\hypersetup{colorlinks=true,linkcolor=blue,citecolor=blue,pdfpagemode=UseNone,pdfstartview={XYZ null null 1.00}}

\tikzstyle arrowstyle=[scale=1]

\newcommand{\rainbow}[1]{\mathfrak{R}(#1)}
\newcommand{\sat}[3]{\operatorname{sat}_{#1}\left(#2, \rainbow{#3}\right)}

\theoremstyle{plain}
\newtheorem{theorem}{Theorem}

\newtheorem{claim}{Claim}
\newtheorem{conjecture}{Conjecture}
\newtheorem{lemma}[theorem]{Lemma}

\newtheorem{proposition}[theorem]{Proposition}

\newtheorem{corollary}[theorem]{Corollary}

\newtheorem{question}{Question}

\title{Rainbow saturation of graphs}

\author[A.~Gir\~{a}o]{Ant\'onio Gir\~{a}o}
\address[Ant\'onio Gir\~{a}o]{School of Mathematics, University of Birmingham, 
Edgbaston, Birmingham, B15 2TT, United Kingdom.
}
\email{giraoa@bham.ac.uk}

\author[D.~Lewis]{David Lewis}
\address[David Lewis]{\,Department of mathematical sciences, University Of Memphis, Memphis, TN 38152, USA.}
\email{david.lewis@gmail.com}

\author[K.~Popielarz]{Kamil Popielarz}
\address[Kamil Popielarz]{\,Department of mathematical sciences, University Of Memphis, Memphis, TN 38152, USA.}
\email{kamil.popielarz@gmail.com}

\date{\today}

\begin{document}
\maketitle
\onehalfspacing

\begin{abstract}
In this paper we study the following problem proposed by Barrus, Ferrara, Vandenbussche, and Wenger. 
Given a graph $H$ and an integer $t$, what is $\operatorname{sat}_{t}\left(n, \mathfrak{R}(H)\right)$, the minimum number of edges in a $t$-edge-coloured graph $G$ on $n$ vertices such that $G$ does not contain a rainbow copy of $H$, but adding to $G$ a new edge in any colour from $\{1,2,\ldots,t\}$ creates a rainbow copy of $H$?
Here, we completely characterize the growth rates of $\operatorname{sat}_{t}\left(n, \mathfrak{R}(H)\right)$ as a function of $n$, for any graph $H$ belonging to a large class of connected graphs and for any $t\geq e(H)$. 
This classification includes all connected graphs of minimum degree $2$. 
In particular, we prove that $\operatorname{sat}_{t}\left(n, \mathfrak{R}(K_r)\right)=\Theta(n\log n)$, for any $r\geq 3$ and $t\geq {r \choose 2}$, thus resolving a conjecture of Barrus, Ferrara, Vandenbussche, and Wenger. 
We also pose several new problems and conjectures.
\end{abstract}

\section{Introduction}
In extremal graph theory, over many decades, much attention has been paid to the following two types of question.
One is the classical Tur\'an-type problem \cite{turan} which asks for the maximum number of edges a graph on $n$ vertices can have provided it does not contain as a subgraph any member of a fixed class of graphs $\mathcal{H}$.
The other question is concerned with another extreme, namely to determine the minimum number of edges in a graph $G$ on $n$ vertices which is $\mathcal{H}$-free but for which the addition of any edge between two non-adjacent vertices of $G$ creates a copy of some graph $H\in\mathcal{H}$.
A maximal (with respect to inclusion) $\mathcal{H}$-free graph $G$ is said to be $\mathcal{H}$-\emph{saturated}. The latter question can then be reformulated: what is the smallest number of edges in a $\mathcal{H}$-saturated graph on $n$ vertices? This number, usually denoted by $\operatorname{sat}(n, \mathcal{H})$, was studied by Zykov \cite{zyk} and independently by Erd\H{o}s, Hajnal, and Moon \cite{ehm} who proved that $\operatorname{sat}(n, K_r)=(r-2)(n-1) -{r-2 \choose 2}$. Soon after, Bollob\'as \cite{Bollobas1} showed that $\operatorname{sat}(n, K_s^{\ell})={n \choose \ell}-{n-(s-\ell) \choose \ell}$, where $K_s^{\ell}$ is the complete $\ell$-uniform hypergraph on $s$ vertices and he conjectured $\operatorname{sat}(n, \mathcal{H})=O(n)$, for any class of graphs $\mathcal{H}$. K\'aszonyi and Tuza~\cite{kt}, in $1986$ confirmed this conjecture. For more information on saturation numbers we refer the reader to the survey of Faudree, Faudree, and Schmitt \cite{ffs}.

In the present paper, we will be interested in a variation of the saturation numbers, following the approach of Hanson and Toft~\cite{ht}, who extended this notion to edge-coloured graphs. We need introduce some definitions first. We define a $t$-edge-coloured graph to be an ordered pair $(G,c)$, where $G$ is a graph and $c$ is a $t$-edge-colouring of $G$, i.e., function from the edge set of $G$ to the set $\{1,2,3, \ldots,t\}$, whose elements we call colours. An edge-coloured subgraph of $G$ is a pair $(H,c|_{E(H)})$, where $H$ is any subgraph of $G$. Throughout the paper, we will usually identify the coloured graph $(G,c)$ with the graph $G$, especially when it is clear from the context which colouring is being used. Note that we do not require edge colourings to be proper. Given an integer $t$ and a family $\mathcal{F}$ of $t$-edge-coloured graphs, we say that a $t$-edge-coloured graph $(G,c)$ is $(\mathcal{F},t)$-saturated if $(G,c)$ contains no member of $\mathcal{F}$ as an edge-coloured subgraph, but the addition of any non-edge in any colour from the set $\{1,2,\ldots,t\}$ creates a copy of a coloured graph in $\mathcal{F}$. Similarly to the usual saturation problem, one denotes by $\operatorname{sat}_t(n,\mathcal{F})$ the minimum number of edges in a $(\mathcal{F},t)$-saturated $t$-edge-coloured graph on $n$ vertices. In \cite{ht}, Hanson and Toft proved that for any sequence of positive integers $2 \leq k_1\leq k_2\leq \ldots \leq k_m$,
\begin{align*}
\operatorname{sat}_t(n, \mathcal{M}(K_{k_1},K_{k_2}\ldots, K_{k_m}))=
\begin{cases}
{n \choose 2} &\text{ if } n\leq k-2m \\
{k-2m \choose 2}+(k-2m)(n-k+2m) &\text{ if } n>k-2m,
\end{cases}
\end{align*}
where $k=\sum\limits_{i=1}^{t} k_i$ and $\mathcal{M}(K_{k_1},K_{k_2}\ldots, K_{k_m})$ is the collection of coloured graphs consisting of a monochromatic copy of $K_{k_i}$ in colour $i$, for each $i\in \{1,2,\ldots,m\}$.

In this paper, we investigate some problems proposed by Barrus, Ferrara, Vandenbussche, and Wenger \cite{sat}. Given a graph $H$ and $t\geq e(H)$, we let $\mathfrak{R}(H)$ to be the collection of all rainbow copies of $H$, i.e. all $t$-edge-coloured graphs $(H,c)$ where each edge is assigned a different colour from $\{1,2\ldots,t\}$.
We shall call $\sat{t}{n}{H}$ the \emph{$t$-rainbow saturation number of $H$}, and, if the set of colours is infinite (say the set of natural numbers) we shall simply write $\sat{}{n}{H}$ and call it the \emph{rainbow saturation number of $H$}. 
Our goal throughout the paper is to determine the value of $\sat{t}{n}{H}$ for a fixed graph $H$.

The authors of \cite{sat} proved several beautiful and surprising results concerning these numbers. 
In particular, they showed a rather interesting phenomenon, namely that there are graphs whose $t$-rainbow saturation numbers grow considerably faster as a function of $n$ then the usual saturation numbers. 
For example, they proved that for every integer $r$ and $t \geq {r \choose 2}$ there exist two positive constants $c_1,c_2$ such that
\[
    c_1 \frac{n\log{n}}{\log{\log{n}}} \leq \sat{t}{n}{K_r}\leq c_2n\log{n}.
\]
In the same paper, the authors determined the $t$-rainbow saturation number of stars, showing that $\sat{t}{n}{K_{1,k}}=\Theta(n^2)$ for any positive integers $t\geq k\geq 2$. This result confirms that the growth rates of rainbow saturation numbers behave very differently from the usual saturation numbers. 
They also state the following conjecture.
\begin{conjecture}[\cite{sat}] \label{conjecture:complete}
    For any integers $r$ and $t$ with $t\geq {r \choose 2}$, $\sat{t}{n}{K_{r}}=\Theta(n\log{n})$.
\end{conjecture}
One of our aims in this paper is to prove this lovely conjecture.
Moreover, we show that any graph $H$ without isolated vertices satisfying $\sat{t}{n}{H}=\Theta(n^2)$, for some $t\geq e(H)$, must be a star.
This answers a question posed in \cite{sat} asking if stars were the only graphs with quadratic $t$-rainbow saturation numbers.
Observe that the function $\sat{t}{n}{H}$ is monotonically decreasing in $t$ for every graph $H$. Therefore, one just needs to show $\sat{t}{n}{H}=o(n^2)$ when $t=e(H)$. Indeed, we show the following stronger result.

\begin{theorem}\label{thm:generalupperbound}
    Let $H$ be a graph without isolated vertices which is not a star.
    Then, for any $t\geq e(H)$, 
    \[
        \sat{t}{n}{H}=O(n\log n).
    \]
\end{theorem}
Observe trivially that the addition of isolated vertices does not change the rainbow saturation numbers for all $n$ sufficiently large. 

Given a graph $H$, we say that a vertex $x\in V(H)$ is \emph{conical} if its degree is $|H|-1$ and we say an edge is \emph{pendant} if one of its endpoints has degree $1$.
For any $r\geq 4$, we define \emph{$K_{r}$ with a rotated edge} to be the graph obtained by taking with a copy of $K_{r}$, adding a new vertex, and "rotating" one edge by replacing one of its endpoints with the new vertex, as in Figure~\ref{fig2}.\\

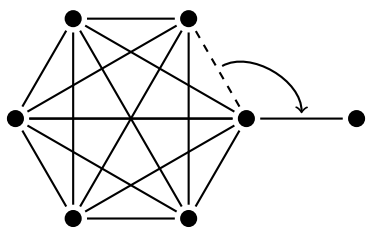
\begin{figure}[htbp]
\centering
\begin{tikzpicture}
[
  every node/.style={draw,circle,inner sep=2pt},
  fsnode/.style={fill=black},
  ssnode/.style={fill=black},
  every fit/.style={ellipse,draw,inner sep=-2pt,text width=2mm},
  shorten >= 2pt,shorten <= 2pt, auto, thick]

\foreach \a in {1,2,...,6}{
\node [fsnode, anchor=center] at ({360/6 * (\a-1)}:1.5 cm+ 1)(f\a) {};
}

\begin{scope}[xshift=30mm,yshift=0cm,start chain=going right,node distance=10mm]
  \node[ssnode,on chain] (s1) {};
\end{scope}

\draw[black](f1) -- (f2) [dashed];
\coordinate (a) at ($(f1)!0.5!(f2)$);
\draw[black](f1) -- (s1);
\coordinate (b) at ($(f1)!0.5!(s1)$);
\draw[->, bend right=90] (a) to [out=60, in=120] (b);
\draw[black](f1) -- (f3);
\draw[black](f1) -- (f4);
\draw[black](f1) -- (f4);
\draw[black](f1) -- (f5);
\draw[black](f1) -- (f6);
\draw[black](f2) -- (f3);
\draw[black](f2) -- (f4);
\draw[black](f2) -- (f5);
\draw[black](f2) -- (f6);
\draw[black](f3) -- (f4);
\draw[black](f3) -- (f5);
\draw[black](f3) -- (f6);
\draw[black](f4) -- (f5);
\draw[black](f4) -- (f6);
\draw[black](f5) -- (f6);
\end{tikzpicture}

\caption{$K_6$ with a rotated edge. The dashed line represents the removed edge.}
\label{fig2}
\end{figure}

In the next result, we completely characterize the growth rates of $t$-rainbow saturation numbers of every connected graph $H$ with no leaves, for every $t\geq e(H)$. 
Actually, we prove a slightly stronger result.

\begin{theorem} \label{thm:class}
Let $H$ be a connected graph of order at least $3$. Then, for every $t\geq e(H)$, $\sat{t}{n}{H}$ equals:
\begin{enumerate}
\item \label{thm:class:star} $\Theta(n^2)$, if $H$ is a star.
\item \label{thm:class:conical} $\Theta(n\log n)$, if $H$ has a conical vertex but is not a star.
\item \label{thm:class:triangular} $\Theta(n\log n)$, if every edge of $H$ is in a triangle.
\item \label{thm:class:linear} $\Theta(n)$, if $H$ contains a non-pendant edge which does not belong to a triangle.
\item \label{thm:class:Hkl} $\Theta(n)$, if $H$ is a $K_{r}$ with a rotated edge, for some even $r \ge 4$.
\end{enumerate}
\end{theorem}

We note that if $H$ is connected with no pendant edges, then, for any $t\geq e(H)$, $\sat{t}{n}{H}=\Theta(n\log n)$ if every edge belongs to a triangle (by \ref{thm:class:triangular}) and $\sat{t}{n}{H}=\Theta(n)$ otherwise (by \ref{thm:class:linear}).

We confirm Conjecture~\ref{conjecture:complete} as a direct consequence of Theorem~\ref{thm:class}:
\begin{theorem}
    For any integers $r$ and $t$ with $t\geq {r \choose 2}$, $\sat{t}{n}{K_{r}}=\Theta_{r,t}(n\log{n})$.
\end{theorem}
We would like to note that Conjecture~\ref{conjecture:complete} was independently proved by Kor\'andi~\cite{Korandisat} and by Ferrara, Johnston, Loeb, Pfender, Schulte, Smith, Sullivan, Tait, and Tompkins~\cite{ferrara2017edgecolored}.

It is easy to check that all graphs excluded from the classification of Theorem~\ref{thm:class} can be constructed by starting with a connected graph in which every edge lies in a triangle and adding pendant edges to the graph.
Note that not all graphs constructed in this way are excluded, as the class of such graphs includes all cliques with a rotated edge and some graphs with a conical vertex.
For simplicity, we denote by $\mathcal{B}$ the class of all connected graphs excluded from the classification of Theorem~\ref{thm:class}.

Although we have not determined the correct order of magnitude of the $t$-rainbow saturation numbers of any graph $H$ in $\mathcal{B}$ for all $t\geq e(H)$, in almost all cases, we were able to determine the order of magnitude of $\sat{t}{n}{H}$ for all sufficiently large values of $t$.
The authors of \cite{sat} also showed that if $H$ is a graph on at least five vertices with a leaf whose neighbour is not a conical vertex and the rest of the vertices do not induce a clique then for any $t \ge \binom{|H|-1}{2}$ we have $\sat{t}{n}{H} = \Theta(n)$.
Our next result covers almost all the remaining graphs containing a pendant edge. 
We show that for every $H$ in $\mathcal{B}$ (with the exception of $K_r$ with a rotated edge, $r$ odd), the $t$-rainbow saturation number of $H$ is linear in $n$, for all $t$ sufficiently large.

\begin{theorem}\label{thm:largenumbercolours}
Let $H$ be a connected graph with no conical vertex and containing at least one pendant edge.
Moreover, suppose $H$ is not a copy of $K_r$ with a rotated edge for  odd $r\geq 5$.
Then, for every $t\geq |H|^{2}$,
\[
    \sat{t}{n}{H}=\Theta(n).
\]
\end{theorem}

In all results discussed above, we assumed that the number of available colours, $t$, is always fixed and does not grow with $n$.
In Theorem~\ref{thm:randomcomplete} we scratch the surface of the case when $t = t(n)$ grows with $n$ and prove that for any $r \ge 3$ there exists a constant $c_{r} > 0$ such that, for any $t \ge \binom{r}{2}$, we have
\[
    \sat{t}{n}{K_{r}} \le \max\left\{ \frac{c_{r}}{\log t} n \log n, 2(r-2)n \right\}.
\]

In particular, this shows (by taking $t(n)$ to be at least linear in $n$) that $\sat{}{n}{K_{r}} = \Theta(n)$, for any $r \ge 3$.

Finally, we shall remark that we did not rule out the existence of a `sharp threshold' for some connected graph $H$, i.e., a $t \ge e(H)$ such that $\sat{t+1}{n}{H}=o(\sat{t}{n}{H})$ as a $n\rightarrow \infty$.
However, if such graph exists it must belong to $\mathcal{B}$, by Theorem~\ref{thm:class}. 
Note also that the set of connected graphs for which we have \emph{not} determined the correct growth rate of their $t$-rainbow saturation numbers for large enough $t$ consists exactly of the aforementioned $K_{r}$'s with a rotated edge for odd $r\geq 5$.

\section{Organization and notation}

In section~\ref{sec:lowerbound}, we prove lower bounds for the $t$-rainbow saturation number of two classes of graphs, namely graphs where every edge belongs to a triangle and graphs which contain a conical vertex, allowing us to establish the correctness of Conjecture~\ref{conjecture:complete}. 
In Section~\ref{sec:upperboundsconn}, we shall prove Theorem~\ref{thm:generalupperbound} when restricted to the class of connected graphs, as well as the main parts of the proof of Theorem~\ref{thm:class} and Theorem~\ref{thm:largenumbercolours}. We split the argument in the following way. 
First, in Subsection~\ref{subsec:trees_2_conn}, we show item \ref{thm:class:linear} of Theorem~\ref{thm:class} and in Subsection~\ref{subsec:cycle}, we prove Theorem~\ref{thm:generalupperbound} assuming the graph is connected.
Secondly, in Subsection~\ref{subsec:leaves}, we establish item \ref{thm:class:Hkl}.
In Subsection~\ref{subsec:complete}, we shall give upper bounds (depending on $t$), for the $t$-rainbow saturation numbers of complete graphs. We also show that, when the palette of colours is infinite, the rainbow saturation numbers of complete graphs are linear.
In Section~\ref{sec:upperboundsdiscon}, we complete the proof of Theorem~\ref{thm:generalupperbound}, showing it also holds for disconnected graphs without isolated vertices.
In Section~\ref{sec:proofs}, we deduce from the results proved in previous Sections Theorem~\ref{thm:class}~and Theorem~\ref{thm:largenumbercolours}.
Finally, in Section~\ref{sec:openproblems} we make some remarks and propose some conjectures and questions that we would like to be investigated.

The notation we use is mostly standard.
For a graph $G$ we define $e(G)$ to be the number of edges in $G$.
For $S \subseteq V(G)$, we denote by $e(S)$ the number of edges with both endpoints in $S$, and, for $S, T \subseteq V(G)$, we denote by $e(S, T)$ the number of edges with one endpoint in $S$ and the other in $T$. A \emph{non-edge} of $G$ is an edge of $\overline{G}$. Moreover, we say a non-edge in a graph $G$ is $\mathfrak{R}(H)$-saturated if adding $e$ in any colour from the palette of colours understood by the context creates a rainbow copy of $H$.  
Also, if $v$ is a vertex in an edge-coloured graph, we say informally that $v$ \emph{sees} a given colour if it is incident with an edge of that colour.
For any positive integer $k$, we define the \emph{$k$-star} to be the graph $K_{1, k}$.
Moreover, for any positive integer $t$, let $[t] = \{1,2,...,t\}$.
All logarithms are base $2$.

\section{Lower bounds}\label{sec:lowerbound}
In this Section, we show that if a graph possesses certain properties then its $t$-rainbow saturation numbers grow at least as fast as $n\log{n}$.
Before doing so, we will need the following trivial lower bound for the rainbow saturation numbers of a connected graph on at least three vertices.
\begin{lemma} \label{lem:lower_bound_linear}
    If $H$ is a connected graph on at least three vertices then $\sat{}{n}{H} \ge \frac{n-1}{2}$.
\end{lemma}
\begin{proof}
    It is easy to check that if $G$ is an $\rainbow{H}$-saturated graph then it has at most one isolated vertex, hence $e(G) \ge \frac{n-1}{2}$.
    Indeed, observe first that, since $H$ is connected and has at least three vertices, every edge in $H$ has an endpoint with degree at least $2$.
    Therefore, if there are two isolated vertices in $G$, say $x$ and $y$, then adding the edge $xy$ to $G$ with any colour must create a copy of $H$, hence either $x$ or $y$ must have degree at least $1$, which gives a contradiction.
\end{proof}

The following theorem improves a result appearing in \cite{sat} and confirms Conjecture~\ref{conjecture:complete}.
\begin{theorem} \label{thm:lower_bound_triangles}
Let $H$ be a graph in which every edge lies in a triangle, then if $t\geq e(H)$,
\[
    \sat{t}{n}{H} \geq \left(\frac{1}{4t}+o(1)\right)n\log n.
\]
\end{theorem}
\begin{proof}[\textbf{Proof of Theorem}~\ref{thm:lower_bound_triangles}]

For each positive integer $n$, let $(G,c)=(G_n,c_n)$ be a $\mathfrak{R}(H)$-saturated $t$-edge-coloured graph on $n$ vertices and $m=m(n)$ edges. 
Note that, by Lemma~\ref{lem:lower_bound_linear}, $m \ge \frac{n-1}{2}$. Moreover, we must have $d(v)\geq 2$ for all $v\in V(G)$. 

For every colour $i \in \{1,2,\ldots,t\}$ and every vertex $v$, let $d_i(v)$ be the degree of $v$ in the subgraph spanned by the $i$-coloured edges and $m_i$ be the total number of $i$-coloured edges. Now, pick a colour, say $1$, and, for each vertex $v$ and each pair $i<j$ of colours different from $1$, consider the complete bipartite graph $B_{v}^{i,j}$ with parts $S_{v}^i$ and $S_{v}^j$, where, for any colour $k$, $S_{v}^k=\{u\in V(G):uv\text{ is a }k\text{-coloured edge in }G\}$. Since the addition of a new edge to $G$ in colour $1$ must create a rainbow triangle, every non-edge of $G$ must belong to at least one of these bipartite graphs.
Let 
\[
    \left\lbrace X^{i,j}_v\sim \operatorname{Bernoulli}\left(\frac{1}{2}\right):v\in V(G),i<j,i,j\neq 1\right\rbrace
\]
be an independent set of random variables and, for each $v\in V(G)$ and every pair of colours $i<j,i,j\neq 1$, set

\begin{align*}
T_v^{i,j}=
\begin{cases}
  S_{v}^i & \mbox{if } X^{i,j}_v = 0 \\
  S_{v}^j  & \mbox{if } X^{i,j}_v = 1\text{.}
\end{cases}
\end{align*}
Now let $U=V(G)\setminus\bigcup\left\{T_v^{i,j}:v\in V(G), i,j\in[t], 1\notin \{i,j\} \right \}$. Notice that, if $uw$ is a non-edge, then at least one of $u$ and $w$ is not in $U$. $U$ is therefore a clique, so as $G$ has $m$ edges,

\begin{align*}
|U|\leq\sqrt{2m+\frac{1}{4}}+\frac{1}{2}\leq\sqrt{2m}+1\text{.}
\end{align*}
We also have the following lower bound on the expected size of $U$.

\begin{align*}
\mathbb{E}[|U|]=\sum\limits_{v\in V(G)}2^{-(t-2)(d(v)-d_1(v))}
\geq n\cdot 2^{-2(t-2)\frac{(m-m_1)}{n}}
\text{.}
\end{align*}
Combining these inequalities, we have that

\begin{align*}
n\cdot 2^{-2(t-2)\frac{(m-m_1)}{n}}\leq\sqrt{2m}+1\text{.}
\end{align*}
Since this holds for every colour, by taking the average over all colours, we obtain

\begin{align*}
n\cdot 2^{-2\frac{(t-1)(t-2)}{tn}m}
\leq\sqrt{2m}+1\text{.}
\end{align*}
Let $\gamma$ be a constant such that $m<(\gamma+o(1))n\log n$. Then $m=n^{1+o(1)}$ and

\begin{align*}
\sqrt{2m}+1=m^{\frac{1}{2}+o(1)}\geq
n\cdot 2^{-2\frac{(t-1)(t-2)}{tn}m}
\geq n^{1-2\gamma\frac{(t-1)(t-2)}{t}+o(1)}
\implies\\
n^{1-2\gamma\frac{(t-1)(t-2)}{t}+o(1)}\leq m^{\frac{1}{2}+o(1)}=n^{\frac{1}{2}+o(1)}\implies\\
1-2\gamma\frac{(t-1)(t-2)}{t}\leq \frac{1}{2}\implies
\gamma\geq\frac{t}{4(t-1)(t-2)} \ge \frac{1}{4t}.
\end{align*}
\end{proof}

Using a similar argument we can show that every graph with a conical vertex also has large $t$-rainbow saturation numbers.
\begin{theorem} \label{thm:max_degree}
    If $H$ is a graph with a conical vertex and $|H|\geq 3$, then, for any ${t\geq e(H)}$,
    \[
        \sat{t}{n}{H} \ge \left(\frac{1}{4t^2}+o(1)\right)n\log n.
    \]
\end{theorem}
\begin{proof}[\textbf{Proof of Theorem}~\ref{thm:max_degree}]
Let $H$ be a graph which is not a star containing a conical vertex $v$. For every positive integer $n$, let $(G,c)=(G_n,c_n)$ be an $\mathfrak{R}(H)$-saturated $t$-edge-coloured graph. As $G$ has at most one isolated vertex, we can find a set $S\subset V(G)$ of size at least $\frac{n-1}{t}$ such that every vertex in $S$ sees the same colour, say colour $1$. Now, we claim that for every non-edge $xy$, with $x,y\in S$, there must exist a rainbow path of length $2$ between $x$ and $y$ using colours in $\{2,3,\ldots,t\}$. Suppose, for a contradiction, this is not the case. When $e=xy$ is added and coloured $1$, we must create a copy $H'$ of $H$, which implies one of the endpoints of $e$ (say $x$) must play the role of $v$ and the other (say $y$) plays the role of a leaf in $H$, the latter must hold by the assumption that there is no rainbow path of length $2$ between $x$ and $y$. However, in this case, there would already exist a rainbow copy of $H$ in $G$, namely $H'\setminus\{y\} \cup \{z\}$, where $z$ is a neighbour of $x$ with the edge $xz$ coloured $1$. We may now apply the same technique used in the proof of Theorem~\ref{thm:lower_bound_triangles}. Let $m$ be the number of edges of $G$. 

As before, for each vertex $x\in G$ and each pair $i<j$ of colours other than $1$, we consider the complete bipartite graph $B_{x}^{i,j}$ with parts $S_{x}^i$ and $S_{x}^j$, where, for any colour $k$, $S_{v}^k=\{u\in S : uv \text{ is a }k\text{-coloured edge in }G\}$. Since every non-edge between vertices of $S$ is joined by a rainbow path in colours other than $1$, each of them is covered by at least one of these bipartite graphs. Let $\left\lbrace X^{i,j}_x\sim \operatorname{Bernoulli}\left(\frac{1}{2}\right):x\in V(G),i<j,i,j\neq 1\right\rbrace$ be an independent set of random variables and, for each $x\in V(G)$ and each pair of colours $i<j,i,j\neq 1$, set

\begin{align*}
T_v^{i,j}=
\begin{cases}
  S_{v}^i & \mbox{if } X^{i,j}_v = 0 \\
  S_{v}^j  & \mbox{if } X^{i,j}_v = 1\text{.}
\end{cases}
\end{align*}
Let $S\setminus\bigcup\left\{T_v^{i,j}|v\in V(G), i,j\in[t], i,j\neq 1\right\}$. If $uw$ is a non-edge, then at least one of $u$ and $w$ is not in $U$. Hence $U$ is a clique, so $|U|\leq m^{\frac{1}{2}+o(1)}$. We also have

\begin{align*}
\mathbb{E}[|U|]=\sum\limits_{u\in S}2^{-(t-2)(d(u)-d_1(u))}\geq
\sum\limits_{u\in S}2^{-(t-2)(d(u)-1)}\geq\\
|S|\cdot 2^{-(t-2)\frac{2e(S)+e(S,V(G)\setminus S)}{|S|}-1}\geq
\frac{n-1}{t}\cdot 2^{-2t(t-2)\frac{m}{n-1}-1}.
\end{align*}

Where the second inequality holds by convexity of $2^{-x}$.
Suppose $\gamma$ is a constant for which $m<(\gamma+o(1))(n-1)\log (n-1)$, then

\begin{align*}
(n-1)^{\frac{1}{2}+o(1)}=m^{\frac{1}{2}+o(1)}\geq\frac{n-1}{t}\cdot 2^{-2t(t-2)\frac{m}{n-1}-1}\geq\\
\frac{n-1}{2t}\cdot 2^{-2t(t-2)\gamma\log(n-1)}=
(n-1)^{1-2t(t-2)\gamma+o(1)}
\text{,}
\end{align*}
which implies that $\gamma\geq\frac{1}{4t(t-2)}$. Therefore
\begin{align*}
    m
    \ge \left(\frac{1}{4t(t-2)}+o(1)\right)(n-1)\log(n-1) 
    \ge \left(\frac{1}{4t^2}+o(1)\right)n\log n.
\end{align*}
\end{proof}

\section{Upper bounds for connected graphs}\label{sec:upperboundsconn}
Throughout this section we will assume all graphs are connected and have at least three vertices.
The aim of this section is to provide constructions of rainbow saturated graphs which, in some cases, are optimal up to multiplicative constants.

First, we show that if $H$ has a cycle then $\sat{t}{n}{H} \le O(n \log{n})$, for any $t \ge e(H)$.
Next, for any graph $H$ with a non-pendant edge not contained in any triangle, we give constructions of $t$-coloured graphs on $n$ vertices and with $\Theta(n)$ edges which are $\rainbow{H}$-saturated.
Observe that if $H$ is not a star then either $H$ contains a cycle or $H$ is a tree which has a non-pendant edge, hence by the aforementioned results $\sat{t}{n}{H} \le O(n \log n)$ for any $t \ge e(H)$.
This answers a question from \cite{sat} for connected graphs, namely that stars are the only connected graphs with quadratic rainbow saturation numbers.
We also provide constructions of $\rainbow{K_{r}}$-saturated graphs on $t$ colours, when $t$ is a function of $n$.

\subsection{Graphs with a non-pendant edge not in a triangle} \label{subsec:trees_2_conn}
In this subsection, we show that if $H$ is a graph with a non-pendant edge not contained in any triangle then for any integers $t \ge e(H)$, $n\ge 1$ we have $\sat{t}{n}{H} \le c_{H} n$, where $c_{H}$ depends only on $H$.

Let $H$ be a connected graph on $p \ge 3$ vertices and $m$ edges and slet $e = xy \in E(H)$ be an edge which is not contained in a triangle.
For $n \ge |H|\cdot e(H)$, we shall construct a graph $G = G_{n, H, e}$ on $n$ vertices together with an edge colouring ${c = c_{n, H, e} : E(G) \rightarrow [m]}$ such that the vast majority of the non-edges of $(G, c)$ are $\rainbow{H}$-saturated and, if $H$ satisfies some additional conditions, $(G, c)$ is $\rainbow{H}$-free. Observe that our coloured graph $ (G, c)$ uses exactly $m=e(G)$ colours, therefore any rainbow copy of $H$ in $G$ must use all these colours. 

First, let $\left\{ e_{1}, \dots , e_{m} = e \right\}$ and $\left\{ v_{1}, \dots, v_{p-1} = x, v_{p} = y \right\}$ be enumerations of the edges and vertices of $H$, respectively.
For every $i \in [m]$, let $H_{i}$ be a copy of $H \setminus \left\{ x, y \right\}$ with the vertex set $V_{i} = \left\{ v_{1}^{i}, \dots, v_{p-2}^{i} \right\}$, where $v_{j}^{i}$ in $H_{i}$ corresponds to $v_{j}$ in $H$.

Now, define a graph $G = K \cup L$ where $G[K] = H_{1} \cup \dots \cup H_{m}$ is a disjoint union of $H_{i}$'s and $L$ is an independent set of size $n - |K|$.
Moreover, for every $u \in L$, $u$ is joined with $v_{j}^{i} \in K$ if and only if either $xv_{j}$ or $yv_{j}$ is an edge in $H$.

Having defined $G$, let us define an edge colouring $c$ of $G$.
Let $w_{1}w_{2}$ be an edge in $G$.
Since $L$ is independent we may assume that $w_{1} = v_{j}^{i}$, for some $i\in [m]$ and $j\in [p-2]$.
Consider now two cases depending on which part $w_{2}$ belongs:
\begin{enumerate}
    \item if $w_{2} \in K$, then $w_{2} = v_{k}^{i}$ for some $k\in [p-2]$, we let $s$ be such that $e_{s} = v_{j}v_{k}$;
    \item if $w_{2} \in L$ , we let $s$ be such that $e_{s} = xv_{j}$ or $e_{s} = yv_{j}$.
\end{enumerate}
It follows from the fact that $e$ is not in a triangle that $s$ is well defined. 
We then define $c(w_{1}w_{2}) = s$ if $s \neq i$ and $c(w_{1}w_{2}) = m$ otherwise.

First, we shall show that every non-edge in $L$ is $\rainbow{H}$-saturated.
\begin{proposition}
    Every non-edge in $L$ is $\rainbow{H}$-saturated.
\end{proposition}
\begin{proof}
    Take any non-edge $w_{1}w_{2}$ in $L$ and any colour $i \in [m]$.
    It is easy to check that adding the $i$-coloured edge $w_{1}w_{2}$ to the graph creates a rainbow copy of $H$ in $\left\{ w_{1}, w_{2}, H_{i} \right\}$.
\end{proof}

Now we shall describe the properties $H$ must have if there exists  a rainbow copy of $H$ in $\left(G, c \right)$.
\begin{lemma} \label{lemma:each_vertex_appears}
    Let $W$ be a rainbow copy of $H$ in $\left( G, c \right)$. Then, all the following must hold.
    \begin{enumerate}
        \item \label{lemma:all_edges_appear} If $v_{i}v_{j}$ is an edge of $H$, for some $i, j \in [p-2]$, then there is $k$ such that $v_{i}^{k}v_{j}^{k}$ is an edge in $W$.
        \item \label{lemma:only_one_duplicate} There is exactly one $i \in [p-2]$ such that there exist distinct $k, k'$ with $v_{i}^{k}, v_{i}^{k'} \in W$ (we shall say that $i$ is \emph{not unique in} $W$).
        \item \label{lemma:unique_downstairs} There is exactly one vertex in $W$, say $z$, such that $z \in L$.
        \item \label{lemma:all_adjacent_to_z} If $v_{i}^{k} \in W$ and $v_{i}$ is adjacent to $x$ or $y$ in $H$ then $v_{i}^{k}$ is adjacent to $z$ in $W$.
        \item \label{lemma:degree} $d_{W}(z) = d_{H}(x) + d_{H}(y) - 1$.
        \item \label{lemma:edges_are_unique} If $v_{i}^{k}v_{j}^{k} \in E(W)$ and $v_{i}^{k'}v_{j}^{k'} \in E(W)$ then $k = k'$.
    \end{enumerate}
\end{lemma}
\begin{proof} 
    For every $k \in [m]$, we let $f_{k} \in E[W]$ be the edge of $W$ of colour $k$. 
    Observe, that for every $k \in [m-1]$, the only $k$-coloured edges in $(G, c)$ are exactly those edges which are `copies' of $e_{k}$, in other words,
    
    \begin{enumerate}
        \item[(a)] if $e_{k} = v_{i}v_{j}$, for $i,j \in [p-2]$ then $f_{k} = v_{i}^{k'}v_{j}^{k'}$ for some $k' \neq k$;
        \item[(b)] and if $e_{k} = v_{i}v_{j}$, for $i \in [p-2]$, $j \in \left\{ p-1, p \right\}$, then $f_{k} = v_{i}^{k'}z$, for some $z \in L$, $k' \neq k$.
    \end{enumerate}
    
    Note that since $H$ is connected and $W$ must intersect at least two distinct $H_i$'s, it follows that $|W \cap L| \ge 1$. Moreover, it follows from $(a)$ and $(b)$ that for every $i \in [p-2]$, there exists some $k'\in [m]$ such that $v_{i}^{k'} \in W$.
    Hence, $(1)$ holds.

    To see $(2)$ and $(3)$, observe first that if there are two different indices $i\neq i'\in [p-2]$ for which there exists two copies of $v_i,v_j$ in $W$ then $|W| \ge (p - 2) + 2 + 1 = p +1$, which is a contradiction.
    Therefore, there is at most one index which is not unique.

    To finish the proof of $(2)$ and $(3)$, it is enough to show that $|W \cap K| \ge p-1$.
    Let us consider where the edge $f_{m}$, of colour $m$ appears in $W$.
    If $f_{m} \in G[K]$, then $f_{m} = v_{i}^{k}v_{j}^{k}$ for some $i, j, k$ such that $v_{i}v_{j} = e_{k}$.
    Since we know by $(a)$, that $f_{k} = v_{i}^{k'}v_{j}^{k'}$ for some $k' \neq k$ we have that both $i$ and $j$ are not unique in $W$, which cannot happen as we have seen. 
    Therefore, we may assume that $f_{m} = zv_{i}^{k}$ for some $z \in L$ and $i, k$.
    By construction $v_{i}$ is adjacent to either $x$ or $y$.
    Without loss of generality, we can assume that $v_{i}$ is adjacent to $x$, and again by construction, $e_{k} = v_{i}x$.
    Since $f_{k} = wv_{i}^{k'}$, for some $w \in L$ and $k'\neq k$, we have that $i$ is not unique in $W$. Hence, $|W \cap K| = p-1$ and $|W \cap L| = 1$ and $w=z$. 

    Now, to prove $(4)$, suppose $v_{i}^{k} \in W$.
    Notice that we already showed that if $i$ is not unique in $W$ then $z$ is adjacent to $v_{i}^{k}$ in $W$.
    Therefore, we may assume that $i$ is unique in $W$.
    Since $v_{i}$ is adjacent to either $x$ or $y$, without loss of generality, we may assume that $v_{i}$ adjacent to $x$, and therefore we have that $v_{i}x = e_{\ell}$ for some $\ell$.
    Hence, as observed before, $f_{\ell} = w v_{i}^{k'}$ for some $w \in L$ and $k' \in [m]$.
    Since there is only one vertex in $L$, namely $z$, and $i$ is unique in $W$ we have that $w = z$ and $k' = k$ hence $f_{\ell} = z v_{i}^{k}$ is an edge in $W$.

    Next, to show $(5)$, note that since $z$ is the only vertex in $W \cap L$, it must be incident with $f_{m}$ and $d_{H}(x) - 1 + d_{H}(y) - 1$ edges of other colours.
    Hence, $d_{W}(z) = d_{H}(x) + d_{H}(y) - 1$.

    Finally, if $(6)$ does not hold then both $i$ and $j$ are not unique in $W$, which contradicts $(2)$.
\end{proof}

\begin{proposition} \label{prop:edge_not_in_triangle}
    Suppose $H$ has an edge $e$ which is in a cycle but not in a triangle then there is no rainbow copy of $H$ in $\left( G, c \right) = \left( G_{n, H, e}, c_{n, H, e} \right)$.
\end{proposition}
\begin{proof}
    Suppose for contradiction that $W$ is a rainbow copy of $H$ in $\left( G, c \right)$.
    Let $g$ be the length of a longest cycle in $H$ which uses $e$.
    We shall show that there is a natural correspondence between the $g$-cycles in $W$ and the $g$-cycles in $H$ not using the edge $e$, thus yielding a contradiction, since the number of $g$-cycles in $W$ is then strictly smaller than the number of $g$-cycles in $H$.

    Let $C$ be a $g$-cycle in $W$.
    We shall find a corresponding $g$-cycle $K_{C}$ in $H$.
    If $C$ does not use vertices from $L$, i.e., $C = v_{k_{1}}^{i}\dots v_{k_{g}}^{i}v_{k_{1}}^{i}$, with $k_{1}, \dots, k_{g} \le p - 2$, then let $K_{C} =  v_{k_{1}} \dots v_{k_{g}}v_{k_{1}}$.
    Note that by construction $K_{C}$ is a $g$-cycle in $H$.
    
    Otherwise, by $(3)$ in Lemma~\ref{lemma:each_vertex_appears}, $C$ uses exactly one vertex from $L$, i.e., $C = uv_{k_{1}}^{i} \dots v_{k_{g-1}}^{i}u$ with $u \in L$ and $k_{1}, \dots, k_{g-1} \le p-2$.
    In that case let $K_{C} = w v_{k_{1}} \dots v_{k_{g-1}}w$, where $w = x$ if $v_{k_{1}}$ is a neighbour of $x$ in $H$, or $w = y$ otherwise.

    We claim that $K_{C}$ is a $g$-cycle in $H$.
    Indeed, observe first that by construction $v_{k_{1}} \dots v_{k_{g-1}}$ is a path in $H$.
    Note also that $v_{k_{1}}$ and $v_{k_{g-1}}$ both have exactly one neighbour in $\left\{ x, y \right\}$. 
    Therefore, if $v_{k_{1}}$ and $v_{k_{g-1}}$ are both adjacent to the same vertex $w \in \left\{ x, y \right\}$ then $K_{C}$ is indeed a $g$-cycle. 
    We can therefore assume, without loss of generality, that $k_{1}$ is adjacent to $x$ and $k_{g-1}$ is adjacent to $y$.
    We note that $k_{1}, \dots, k_{g_{1}}, x, y$ is then a $(g+1)$-cycle in $H$ using the edge $e = xy$, which contradicts the assumption that $g$ is the size of a longest cycle in $H$ using the edge $e$.
    
    It is easy to check now that if $K_{C} = K_{C'}$, for some $g$-cycle $C'$ (different from $C$) in $W$, then we obtain a contradiction to $(6)$ of Lemma~\ref{lemma:each_vertex_appears}.
    Finally, there is no $g$-cycle $C$ in $W$ such that $K_{C}$ is a $g$-cycle in $H$ using the edge $e$, thus we obtain a contradiction.
\end{proof}

Recall that that an edge is a \emph{bridge} if its removal increases the number of connected components.

\begin{proposition} \label{prop:cut_edge}
    If $H$ has a non-pendant bridge then there is an edge $e \in H$ such that there is no rainbow copy of $H$ in $\left( G_{n, H, e}, c_{n, H, e} \right)$.
\end{proposition}
\begin{proof}
    If there is an edge $e'$ in $H$ which is in a cycle but not in a triangle then the result follows from Proposition~\ref{prop:edge_not_in_triangle}, by taking $\left( G_{n, H, e'}, c_{n, H, e'} \right)$. 
   Hence, we may assume that every edge in $H$ which is not in a triangle is a bridge.
    Let $e = xy$, with $d(x) \ge d(y)$, be a non-pendant bridge in $H$ for which $d(x)$ is maximized.
    By the assumption $e$ is well defined.

    Suppose for contradiction that $W$ is a rainbow copy of $H$ in $(G, c)$.
    We will show that the number of non-pendant bridges in $W$ is strictly smaller than the number of non-pendant bridges in $H$, thus obtaining a contradiction.

    Observe first that we cannot have a non-pendant bridge in $W$ incident with any vertex $z \in L$ as then, by $(5)$ of Lemma~\ref{lemma:each_vertex_appears}, we have $d(z) \ge d(x) + d(y) - 1 \ge  d(x) + 1$ which would contradict the maximality of $d(x)$.
    Therefore, if there is a non-pendant bridge in $W$ then it must be within $K$.

    Let $b = v_{i}^{k}v_{j}^{k}$ be a non-pendant bridge in $W$, for some $i, j, k$.
    We shall show that $e_{b} = v_{i}v_{j}$ must be a non-pendant bridge in $H$.
    By assumption every edge in $H$ which is not in a triangle is a bridge hence $v_{i}v_{j}$ is contained in a triangle, say in $v_{i}v_{j}v_{\ell}$ for some $\ell \in [p]\setminus\{i,j\}$.

    Observe that if $v_{i}v_{j}x$ or $v_{i}v_{j}y$ is a triangle in $H$ then, by $(4)$ of Lemma~\ref{lemma:each_vertex_appears}, $v_{i}^{k}v_{j}^{k}z$ is a triangle in $W$; this contradicts the assumption that $v_{i}^{k}v_{j}^{k}$ is a bridge.
    Therefore we can assume that $v_{i}v_{j}v_{\ell}$ is a triangle with $\ell \le p-2$.
    Since $v_{i}^{k}v_{j}^{k}$ is a bridge in $W$ it follows that the edge cannot belong to any triangle in $W$.
    Therefore either $v_{i}^{k}v_{\ell}^{k}$ or $v_{j}^{k}v_{\ell}^{k}$ is not an edge in $W$.
    Without loss of generality we can assume that $v_{i}^{k}v_{\ell}^{k}$ is not an edge in $W$.
    Hence we must have by $(1)$ of Lemma~\ref{lemma:each_vertex_appears} that, for some $k' \neq k$, $v_{i}^{k'}v_{\ell}^{k'}$ is an edge in $W$.
    By the same lemma, there also must exist $k''$ such that $v_{j}^{k''}v_{\ell}^{k''}$ is an edge in $W$.
    But then there are two indices $i$ and $\ell$ which are not unique in $W$ contradicting $(2)$ of Lemma~\ref{lemma:each_vertex_appears}.
    Therefore, we have that $e_{b} = v_{i}v_{j}$ is indeed a bridge in $H$. 

    Note that by $(6)$ of Lemma~\ref{lemma:each_vertex_appears} we have that $e_{b} \neq e_{b'}$ for distinct non-pendant bridges $b, b'$ in $W$.
    Hence we found a correspondence between the non-pendant bridges in $W$ and the non-pendant bridges in $H \setminus \left\{ e \right\}$, which gives a contradiction as then the number of non-pendant bridges in $W$ is strictly smaller than the number of non-pendant bridges in $H$.
\end{proof}

\begin{theorem} \label{thm:trees_2_conn}
    If $H$ has a non-pendant edge not contained in a triangle then for any integers $t \ge e(H)$ and $n$ we have
    \[
        \sat{t}{n}{H} \le c_{H} \cdot n,
    \]
    where $c_{H} = e(H) \cdot (|H|-2)$.
\end{theorem}

\begin{proof}
    When $n \le e(H) \cdot (|H|-2)$, the result follows easily by considering a monochromatic $K_{n}$.
    We may then assume that $n > e(H) \cdot (|H|-2)$.
    Consider an edge in $H$ as in the statement of Proposition~\ref{prop:edge_not_in_triangle}~or~\ref{prop:cut_edge}.
    Then there is no rainbow copy of $H$ in $\left( G=G_{n, H, e}, c_{n, H, e} \right)$ and every non-edge in $L$ is $\rainbow{H}$-saturated.
    If there are non-edges in $G$ which are not $\rainbow{H}$-saturated for some colour $i$, we can simply add those edges to $G$ and colour them with an appropriate colour, obtaining $\left( G', c' \right)$.
    Note that $e(G') \le |L||K| + \binom{|K|}{2} \le (n-|K|)|K| + |K|^{2} = n|K| \le n \cdot e(H) \cdot (|H|-2)$.
\end{proof}

\subsection{Graphs with a cycle}\label{subsec:cycle}

The construction presented in this subsection will be very similar to the one in Subsection~\ref{subsec:trees_2_conn}.
Let $H$ be a graph on $p$ vertices with a cycle.
Observe that if $H$ is triangle-free then there is an edge in $H$ which is in a cycle but not in a triangle hence by a result from previous section we have that $\sat{t}{n}{H} = O(n)$.
Therefore, we can assume that $H$ has a triangle.
Let $e = xy$ be an edge of $H$ which is contained in a triangle. 

As before, for $n$ large enough we shall construct a graph $G = G^{r}_{n,H, e}$ on $n$ vertices together with an edge colouring ${c = c^{r}_{n, H, e} : E(G) \rightarrow [t]}$ such that the vast majority of the non-edges of $(G, c)$ is $\rainbow{H}$-saturated and $(G, c)$ is $\rainbow{H}$-free.

Let $\left\{ e_{1}, \dots , e_{m} = e \right\}$ and $\left\{ v_{1}, \dots, v_{p-1} = x, v_{p} = y \right\}$ be enumerations of the edges and vertices of $H$, respectively.
For all $i \in [m]$ and $j \in [h]$, where $h = \lceil \log(n^{2}m+1)\rceil$, let $H_{i, j}$ be a copy of $H \setminus \left\{ x, y \right\}$ with the vertex set $V_{i, j} = \left\{ v_{1}^{i, j}, \dots, v_{p-2}^{i, j} \right\}$, where $v_{l}^{i,j}$ in $H_{i,j}$ corresponds to $v_{l}$ in $H$.

Now we define a graph $G = K \cup L$, where $G[K] = \bigcup_{i, j}H_{i,j}$ is a disjoint union of $H_{i, j}$'s and $L$ is an independent set of size $n - |K|$.
Moreover, for every $u \in L$ and $H_{i, j}$ we shall toss a coin and based on the result decide how to join the vertices in $H_{i, j}$ with $u$.
More precisely, for $u \in L$, $i \in [m]$ and $j \in [h]$, let $X_{u, i, j}$ be a random variable such that $\mathbb{P}\{X_{u, i, j} = x\} = \mathbb{P}\{X_{u, i, j} = y\} = \frac{1}{2}$, and let all the $X_{u, i, j}$'s be independent.
Now join $u$ with $v_{k}^{i, j} \in H_{i, j}$ if and only if $v_{k}X_{u,i,j} \in E(H)$.

Having defined $G$ let us define the edge colouring $c$.
Let $w_{1}w_{2}$ be an edge in $G$.
Since there are no edges in $L$ we can assume that $w_{1} = v_{k}^{i, j}$, for some $i, j, k$.
Consider two cases depending on $w_{2}$:
\begin{enumerate}
    \item if $w_{2} \in K$ and $w_{2} = v_{k'}^{i}$ for some $k'$, then let $s$ be such that $e_{s} = v_{k}v_{k'}$;
    \item if $w_{2} \in L$ then let $s$ be such that $e_{s} = v_{k}X_{w_{2}, i, j}$.
\end{enumerate}
Now $c(w_{1}w_{2}) = s$ if $s \neq i$ and $c(w_{1}w_{2}) = m$ otherwise.

\begin{proposition} \label{prop:cycles_saturated}
    With positive probability every non-edge in $L$ is $\rainbow{H}$-saturated.
\end{proposition}
\begin{proof}
    Let $f = uv$ be a non-edge in $L$ and $i \in [m]$ some colour.
    Notice, that if $f$ is $i$-coloured and there is some $j$ for which $X_{u, i, j} \neq X_{v, i, j}$, then we can find a rainbow copy of $H$ in $\left\{ u, v, H_{i, j} \right\}$.
    Call the pair $(uv, i)$ \emph{bad} if $X_{u, i, j} = X_{v, i, j}$ for every $j \in [h]$.
    The probability that $(uv, i)$ is bad is equal to
    \[
        \mathbb{P}\{X_{u, i, j} = X_{v, i, j} \text{ for every } j\} = 2^{-h}.
    \]
    Since we have $\binom{|L|}{2} \le n^{2}$ non-edges in  $L$ and $m$ colours the expected number of bad pairs is
    \[
        \mathbb{E}\left[\#\text{bad pairs}\right] \le 2^{-h} n^{2} m \le \frac{n^{2}m}{n^{2}m + 1} < 1.
    \]
    Therefore with positive probability there is no bad pair, hence with positive probability every non-edge in $L$ is $\rainbow{H}$-saturated.
\end{proof}

\begin{proposition} \label{prop:cycles_no_rainbow}
    There is no rainbow copy of $H$ in $\left( G, c \right)$.
\end{proposition}
\begin{proof}
    Suppose $W$ is a rainbow copy of $H$ in $\left( G, c \right)$.
    We shall show that there is a natural correspondence between the triangles in $W$ and the triangles in $H$ not using the edge $xy$, thus obtaining a contradiction, since the number of triangles in $W$ is then strictly smaller than the number of triangles in $H$.

    Let $T$ be a triangle in $W$.
    We shall find a corresponding triangle $K_{T}$ in $H$.
    If $T$ does not uses vertices from $L$, i.e., $T = \left\{ v_{k_{1}}^{i, j}, v_{k_{2}}^{i, j}, v_{k_{3}}^{i, j} \right\}$, with $k_{1}, k_{2}, k_{3} \le p - 2$, then let $K_{T} = \left\{ v_{k_{1}}, v_{k_{2}}, v_{k_{3}} \right\}$.
    Note that by construction $K_{T}$ is a triangle in $H$.
    
    Otherwise, since $L$ is independent, $T$ uses exactly one vertex from $L$, i.e., $T = \left\{ v_{k_{1}}^{i, j}, v_{k_{2}}^{i, j}, u \right\}$ with $u \in L$ and $k_{1}, k_{2} \le p-2$.
    In that case let $K_{T} = \left\{ v_{k_{1}}, v_{k_{2}}, X_{u, i, j} \right\}$ .
    Again, by construction $K_{T}$ is a triangle in $H$.
    
    It is easy to check now that if $K_{T} = K_{T'}$ for some distinct triangles $T$ and $T'$ in $W$ then at least one colour appears twice in $E(T) \cup E(T')$, which is a contradiction.
    Finally, there is no triangle $T$ in $W$ such that $K_{T}$ is a triangle in $H$ using the edge $xy$.
    This proves that there is no rainbow copy of $H$ in $G$.
\end{proof}

Using those two propositions we are ready to prove the main theorem of this subsection.
\begin{theorem}\label{thm:cycles}
    If $H$ contains a cycle then 
    \[
        \sat{t}{n}{H} \le (1+o_{H}(1))c_{H}\cdot n\log n,
    \]
    where $c_{H} = 2e(H)(|H|-2)$.
\end{theorem}
\begin{proof}
    By Theorem~\ref{thm:trees_2_conn} from the previous subsection we may assume that $H$ contains a triangle.
    Let $e$ be an edge in $H$ contained in a triangle.
    For $n$ large enough, it follows from Propositions~\ref{prop:cycles_saturated}~and~\ref{prop:cycles_no_rainbow}, that there is $\left( G, c \right)$, with vertex partition $K \cup L$ (where $|K| = e(H) \cdot |H| \cdot h$), such that every non-edge in $L$ is $\rainbow{H}$-saturated and there is no rainbow copy of $H$ in $\left( G, c \right)$.
    If there are any non-edges which are not $\rainbow{H}$-saturated we can just add those edges with appropriate colours to $G$ obtaining $\left( G', c' \right)$.
    Therefore $\left( G', c' \right)$ is $\rainbow{H}$-saturated and the number of edges in $G'$ is at most $(n-|K|) \cdot |K| + |K|^{2} = n \cdot |K| = (1+o_{H}(1))c_{H}\cdot n\log n$, where $c_{H} = 2e(H)(|H|-2)$.
\end{proof}

Theorem~\ref{thm:generalupperbound} restricted to the class of connected graphs follows easily as a corollary of the previous theorem and Theorem~\ref{thm:trees_2_conn}. 
\begin{corollary}\label{corol:generalconnected}
Let $H$ be a connected graph on at least three vertices which is not a star.
Then, for every $t\ge e(H)$, 
\[
    \sat{t}{n}{H}=O(n\log n).
\] 
\end{corollary}
\begin{proof}
If $H$ contains a cycle then we are done by Theorem~\ref{thm:cycles}. If not, then $H$ is a tree containing a non-pendant edge and the result follows from Theorem~\ref{thm:trees_2_conn}. 
\end{proof}

\subsection{Graphs with leaves} \label{subsec:leaves}

In this subsection we are concerned with connected graphs which contain a leaf.
In \cite{sat}, Barrus, Ferrara, Vandenbussche, and Wenger showed that, with few exceptions, if a connected graph $H$ has a leaf, then for $t \ge \binom{|H|-1}{2}$, $\sat{t}{n}{H} = \Theta(n)$.
\begin{theorem}[Barrus, Ferrara, Vandenbussche, and Wenger \cite{sat}] \label{thm:bfvw}
    Let $H$ be a graph on at least five vertices with a leaf whose neighbour is not a conical vertex and such that the rest of the vertices do not induce a clique.
    Then, for any $t \ge \binom{|H|-1}{2}$, we have ${\sat{t}{n}{H} = \Theta(n)}$.
\end{theorem}

To prove similar bounds for the remaining connected graphs containing a leaf we shall introduce some terminology. 
We let $H_{k, \ell}$ be the graph obtained by taking a $K_k$ (for some $k\geq 3$) and adding two new vertices $x$ and $y$, where $x$ adjacent to some $\ell$ vertices of the clique and $yx$ is a pendant edge. 
We shall call $x$ the \emph{middle} vertex and $y$ the \emph{leaf} vertex.
Note that all such graphs are isomorphic however we choose the $\ell$ neighbours of $x$ in $K_k$.
Also, observe that the graph $K_k$ with a \emph{rotated} edge is just $H_{k-1,k-2}$. 

The following proposition shows that for any $\ell \le k -2$, the $t$-rainbow saturation number of $H_{k, \ell}$ is linear in $n$ when the number of colours is sufficiently large.
\begin{theorem}\label{thm:Hkl}
    For any $2 \le \ell \le k-2$ and $t \ge k(k-1)$ we have that $\sat{t}{n}{H_{k, \ell}} = O(n)$.
\end{theorem}
\begin{proof}
    Let $G = K \cup L$ where $G[K]$ is a disjoint union of two cliques of size $k$, say $C_{1}$, $C_{2}$, and $L$ is independent set on $n - 2k$ vertices.
    Now, fix $\ell + 1$ vertices $C_{1}$ and $\ell + 1$ vertices of $C_{2}$ and join each vertex in $L$ to all of those vertices.

    Let $A, B \subseteq [k(k-1)]$, with $|A| = |B| = \frac{k(k-1)}{2}$ be a disjoint union of colours and $A' \subseteq A$, $B' \subseteq B$ be any subsets of size $\ell+1$.
    We shall describe the colouring of the edges of $G$.
    First, colour the edges of $C_{1}$ using distinct colours from $A$, and colour the edges of $C_{2}$ using distinct colours from $B$.
    Now, for every vertex $v\in L$ colour the edges incident with $v$ with distinct colours from $B'$ if the edges are incident with $C_{1}$ and distinct colours from $A'$ if the edges are incident with $C_{2}$.
    Note that in this colouring each vertex in $L$ is incident with $2(\ell+1)$ edge of different colours.

    We claim that there is no rainbow copy of $H_{k, \ell}$ in $G$.
    Suppose for contradiction that $W$ is a rainbow copy of $H_{k, \ell}$ in $G$.
    First let us find a copy of $k$-clique $C$ in $W$.
    Up to symmetry there are two cases: either $C$ uses all the vertices from $C_{1}$ or it uses $k-1$ vertices from $C_{1}$ and one vertex from $L$.
    In the former case the middle vertex must be in $L$ and the leaf vertex must be in $C_{2}$.
    Which is a contradiction since $C$ uses all colours of $A$ and the edge between the middle and leaf vertices uses a colour from $A' \subset A$, therefore $W$ is not rainbow.
    In the other case, when $C$ uses a vertex from $L$, say $z$, note that $\ell = k-2$ and therefore the edges between $z$ and the rest of the clique $C$ use all of the colours from $B'$.
    Observe now that the middle vertex cannot be in $C_{2}$ since it has to be adjacent to at least two vertices of the clique $C$ (we assumed that $\ell \ge 2$).
    Also, the middle vertex cannot be in $L$ since it has to be adjacent to at least one vertex from $C_{1} \cap C$, hence must be incident with an edge of colour from $B'$ but all the colours of $B'$ have already been used by the edges incident with $z$.
    Therefore, the middle vertex $z$ must belong to $C_{1}\setminus C$ and the leaf must be in $L$.
    This is impossible as $z$ is not joined to any vertex of $L$, which is a contradiction. 

    Now we shall show that every non-edge in $L$ is $\rainbow{H}$-saturated for any colour $i \in [t]$.
    By symmetry, we can assume that $i \in B$. (If $i\notin A\cup B$ the same argument holds). 
    It is easy to check now that adding the edge $xy$, with $x,y \in L$, and colouring it with colour $i$, we create a rainbow copy of $H_{k, \ell}$ using all the vertices from $C_{1}$ and $x, y$,  where $x$ and $y$ play the roles of the middle and leaf vertices, respectively.
\end{proof}

The following theorem shows that, when $r\geq 4$ is even, the $t$-rainbow saturation of $K_{r}$ with a rotated edge is linear.

\begin{theorem}\label{thm:rotatededge}
Let $r \ge 4$ be even and $H$ be $K_{r}$ with a rotated edge.
Then, for any $t \ge \binom{r}{2}$, $\sat{t}{n}{H} = O(n)$.
\end{theorem}

\begin{proof}
Assume $t=\binom{r}{2}$. We first define a graph $\Gamma$ with vertex set $[r]^\frac{r}{2}$ and an edge between each pair of vertices that differ in exactly one component. Now we will define an edge colouring of $\Gamma$.

We identify the elements of $[t]$ with the edges of $K_r$ (with vertex set $[r]$). It is well known that $K_r$ has a proper edge colouring with $r-1$ colours if $r$ is even. Fix one such colouring $c$. The edges of any given colour $i$ form a matching with $\frac{r}{2}$ edges, and every vertex is incident with exactly one edge of colour $i$. For each $i\in[r-1]$, choose an arbitrary bijection $g_i$ from $[\frac{r}{2}]$ to the set of edges of colour $i$. For each vertex $x$ of $\Gamma$, let $S(x)$ be the sum of the components of $x$ modulo $r$. We define the edge colouring of of $\Gamma$ as follows: If $x$ and $y$ are two vertices of $\Gamma$ that differ in the $k^{th}$ component, colour the edge $xy$ by $g_{c(e)}\left(k+g_{c(e)}^{-1}(e)\right)$, where $e=\left\{S(x),S(y)\right\}$. We claim that every clique in $\Gamma$ is rainbow and that every vertex is incident with exactly one edge of each colour. For the first claim, observe that the restriction of $S$ to a maximal clique is a bijection from the vertices of that clique to those of our $K_r$, and the function $e\mapsto g_{c(e)}\left(k+g_{c(e)}^{-1}(e)\right)$, where $k$ is the component on which all the elements of the clique differ, permutes the edges of $K_r$. For the second claim, let $f$ be any edge of our $K_r$ and let $i=c(f)$ be its colour. Given a vertex $x$ of $\Gamma$, let $v$ be the unique vertex of $K_r$ such that $\{v,S(x)\}$ is coloured $i$. Notice that $x$ has exactly $\frac{r}{2}$ neighbours $y$ such that $S(y)=v$, and each of these neighbours differs from $x$ in a different component, hence each edge $xy$ is a coloured with a different $i$-coloured edge of $K_r$, hence $x$ sees the colour $f$. Therefore every vertex of $\Gamma$ sees every colour. But every vertex of $\Gamma$ has degree $r\choose 2$, so it must be incident with exactly one edge of each colour.

To show that $\Gamma$ is $\mathfrak{R}(H)$-free, we first observe that every clique in $\Gamma$ is a subset of a maximal clique. Hence if there is a rainbow copy of $H$ in $\Gamma$, the "missing" edge of this copy must have the same colour as the pendant edge, contradicting the fact that the colouring of $\Gamma$ is proper.

Now, for any $n$, let $G$ be a graph on $n$ vertices consisting of the disjoint union of $\left\lfloor\frac{n}{r^{\frac{r}{2}}}\right\rfloor$ copies of $\Gamma$ and a monochromatic clique on the leftover vertices. $G$ is $\mathfrak{R}(H)$-free because each of its components is. Suppose we add to $G$ a new edge $e$ in any colour $i$. One endpoint $x$ of this new edge must be in a copy of $\Gamma$. Since $x$ is incident with an edge of colour $i$ and this edge is in a rainbow copy of $K_r$, removing this edge and adding $e$ creates a rainbow $H$. $G$ is therefore an $\mathfrak{R}(H)$-saturated graph with at most $\frac{1}{2}{r\choose 2}r^\frac{r}{2}\left\lfloor\frac{n}{r^{\frac{r}{2}}}\right\rfloor+{r^\frac{r}{2}-1\choose 2}$ edges.
\end{proof}

~\subsection{Complete graphs} \label{subsec:complete}

\begin{theorem} \label{thm:randomcomplete}
    For any $r\geq 3$ there exists a positive constant $c_{r}$ (depending only on $r$) such that the following holds.
    For any $n$ and $t = t(n) \ge \binom{r}{2}$, 
    \[
        \sat{t}{n}{K_r} \leq \max \left\{ \frac{c_{r}}{\log t} n \log n, 2(r-2)n \right\}.
    \]
\end{theorem}

\begin{proof}[\textbf{Proof of Theorem}~\ref{thm:randomcomplete}]

First, it is clear we may assume $n$ is sufficiently large, by taking $c_r$ large enough. Note that if $t \le r^{7}$, by Theorem~\ref{thm:cycles} we have 
\[
    \sat{t}{n}{K_{r}} 
    \le 2\binom{r}{2}r n \log n 
    \le \frac{r^{3} \log r^{7}}{\log t} n \log n
    = \frac{7r^{3} \log r}{\log t} n \log n,
\]
for $n$ large enough, depending only on $r$.
We may then assume that $t \ge r^{7}$.
Let $\ell$ be a positive integer (to be specified later) and $G$ be the union of $2\ell$ disjoint $(r-2)$-cliques together with an independent set $M$ of size $n-2(r-2)\ell$, where each edge with one endpoint in $M$ and the other in one of the cliques is present, and there are no edges between two distinct cliques.
Observe that $G$ does not contain any copies of $K_r$, because any such copy would need to use at least two vertices from $M$.

Let $A, B$ an equipartition of the integers $\{1,2,\ldots,t\}$ (thus, $A, B$ partition $[t]$ and $\left| |A| - |B| \right| \le 1$).
Now, we shall arbitrarily colour the edges of the first $\ell$ $(r-2)$-cliques with the colours from $A$ and the edges of the remaining $\ell$ $(r-2)$-cliques with the colours from $B$, such that in each clique no colour appears twice. 
For each $(r-2)$-clique $K$, let $C_{K}$ be the set colours used by the edges of $K$.
Moreover, for each vertex $x\in M$ and each clique $K$, we shall take a subset $B_{x, K} \subseteq A \setminus C_{K}$, if $C_{K} \subseteq A$, or $B_{x, K} \subseteq B \setminus C_{K}$ otherwise, of size $r-2$ uniformly at random (and independently for every choice of $x$ and $K$) and colour each edge from $x$ to $K$ with a different element of $B_{x, K}$.
Our aim is to prove that with positive probability the addition of any coloured edge between two vertices in $M$ will form a rainbow copy of $K_{r}$. 
To do so, let us compute the probability that some edge $e = xy$, with both endpoints in $M$, coloured $c$ creates a rainbow copy of $K_r$.
By symmetry, we can assume that $c \in B$.
Let $t' = |A| - \binom{r-2}{2}$.
Suppose $K'$ is a rainbow copy of $K_{r-2}$ such that $C_{K'} \subseteq A$.

First, we need the following easy claim.
\begin{claim} \label{claim:prob_sets_dont_intersect}
    For positive integers $s, u$ with $s \ge 2u - 1$ the following holds
    \[
        \frac{\binom{s-u}{u}}{\binom{s}{u}} \ge 1 - \frac{u^{2}}{s-u+1}.
    \]
\end{claim}
\begin{proof}
    Note first, that since $s \ge 2u - 1$ we have $\frac{u}{s-u+1} \le 1$.
    Hence
    \begin{align*}
        \frac{\binom{s-u}{u}}{\binom{s}{u}} 
        &= \frac{(s-u)! (s-u)!}{s! (s-2u)!}
        = \frac{(s-2u+1)\cdot (s-2u+2) \cdots (s-u)}{(s-u+1) \cdot (s-u+2)\cdots s}  \\
        &= \frac{s-2u+1}{s-u+1} \cdot \frac{s-2u+2}{s-u+2} \cdots \frac{s-u}{s} 
        = \left(1- \frac{u}{s-u+1}\right) \left(1-\frac{u}{s-u+2}\right) \cdots \left(1 - \frac{u}{s}\right) \\
        &\ge \left(1 - \frac{u}{s-u+1}\right)^{u}
        \ge 1- \frac{u^{2}}{s-u+1},
    \end{align*}
    where the last inequality follows from Bernoulli's inequality: $(1-x)^{p} \ge 1-px$ for $p \ge 1$ and $x \in [0, 1]$.
\end{proof}

Observe that by construction $c \not\in \left( C_{K'} \cup B_{x, K'} \cup B_{y, K'} \right)$.
Hence as long as $B_{x, K'}$ and $B_{y, K'}$ are disjoint we are done, i.e., there is a rainbow copy of $K_{r}$ in $\left\{ x, y \right\} \cup K'$.
Let us bound the probability that $B_{x, K'}$ and $B_{y, K'}$ are disjoint. 
To do that, we apply Claim~\ref{claim:prob_sets_dont_intersect} with $s = t'$ and $u = r-2$:
\begin{align*}
    \mathbb{P}\left\{B_{x, K'} \cap B_{y, K'} = \emptyset\right\}
    &= \frac{\binom{t'-(r-2)}{r-2}}{\binom{t'}{r}}
    \ge 1 - \frac{(r-2)^{2}}{t' - r + 3}.
\end{align*}
Hence, since $t' \ge t/2 -1$ and $t \ge r^{7}$, we have
\[
    \mathbb{P}\{\{x,  y\} \cup K' \text{ is not rainbow } K_{r} \} = 1 - \mathbb{P}\left\{ B_{x, K'} \cap B_{y, K'} = \emptyset \right\} \le \frac{(r-2)^{2}}{t'-r+3} \le \frac{1}{\sqrt{t}}.
\]
Note, there are $\ell$ rainbow copies of $K_{r-2}$ which only use colours from $A$, so we deduce that
\[
    \mathbb{P}\{e \text{ in colour c does not create a rainbow } K_{r}\} \leq t^{-\ell/2}.
\]
Therefore, the probability that some edge with both endpoints in $M$ is \emph{bad}, i.e. the addition of $e$ in some colour does not form a rainbow copy of $K_{r}$ is at most 
\[
    e(G) \cdot t^{-\ell/2}.
\]
This holds because if we colour $e$ in some colour not appearing in the edges of the graph, then we clearly form a rainbow copy of $K_r$. Hence, taking $\ell = \max \left\{ \left\lceil \frac{10\log n}{\log (t)} \right\rceil, 1\right\}$, we get 
\[
    \mathbb{P}\{\text{ some edge is \emph{bad}}\}\leq e(G) {M \choose 2} t^{-\ell/2}\leq n^{4}{2^{-5\log n}}\leq \frac{1}{n}< 1.
\]
We have thus proved there exists a colouring of $G$ for which no edge with both endpoints in $M$ is \emph{bad}.
If there are still some unsaturated non-edges in $G$, just keep adding them with appropriate colours to $G$.
Let $N = V(G) \setminus M$.
We are done as 
\begin{align*}
    e(G) 
    &\le |N|(n-|N|) + \binom{|N|}{2}
    \leq |N|n - |N|^{2} + |N|^{2}
    \le |N|n \\
    &\leq 2\ell(r-2)n.
\end{align*}
So if $\ell = 1$ then $e(G) \le 2(r-2)n$ and if $\ell > 1$ then  $e(G) \le \frac{20(r-2)}{\log t} n \log n$.
In order for the graph to be well-defined we must take $n$ big enough (depending on $r$ only) so that $2(r-2)\ell \le n$.
\end{proof}

Observe that as long as $t(n) \ge \Omega(n)$ we have $\sat{t}{n}{K_{r}} = \Theta(n)$.
\begin{corollary} \label{cor:randomcomplete}
    For any $r \ge 3$ we have
    \[
        \sat{}{n}{K_{r}} \le 2(r-2)n.
    \]
\end{corollary}
\begin{proof}
    When $n \le 2(r-2)$ then the results follows trivially by considering a monochromatic $K_{n}$.
    We can therefore assume that $n \ge 2(r-2)$.
    Observe that when there is not a restriction on the number of colours then in our construction we can assign each edge a different colour.
    In that case we can take $\ell = 1$, which corresponds to a disjoint union of an independent set $A$ and two $(r-2)$-cliques $B$ and $C$, such that all the edges between $A$ and $B \cup C$ are present, and possibly some edges between $B$ and $C$.
    The number of edges is then at most $2(r-2)n$.
\end{proof}
We conjecture that this bound is best possible up to an additive constant.

The following construction gives a better upper bound for the rainbow saturation numbers of a triangle, at least when $t$ is not too large compared to $n$.
\begin{theorem}
\label{tri}
For any $t\geq 3$ with $t\equiv 1\text{ or }3\pmod 6$, then
\[
    \sat{t}{n}{K_3}\leq \frac{3}{\log {t\choose 2}}n\log n + 3n.
\]
In particular, $\sat{3}{n}{K_3}\leq \frac{3}{\log 3}n\log n + 3n$.
\end{theorem}

\begin{proof}
Let $S$ be a Steiner triple system of order $t$, i.e., a set of three-element subsets of $[t]$ such that every pair of elements of $[t]$ is contained in exactly one element of $S$. We call the elements of $t$ \emph{points} and the elements of $S$ \emph{lines}. It can be shown (see, e.g., Kirkman~\cite{kirkman})
 that such a system exists if and only if $t\equiv 1\text{ or }3\pmod 6$ and that any such system has exactly $\frac{t(t-1)}{6}$ lines. We define a binary operation $\star:[t]^2\to [t]$ as follows:

\begin{align*}
a\star b=
\begin{cases}
  a & \mbox{if } a=b \\
  c, \text{ where }c\text{ is the unique point such that }abc\text{ is a line}  & \mbox{if } a\neq b\text{.}
\end{cases}
\end{align*}

\noindent
This operation has the property that, for every fixed $a$ in $[t]$, the map $b\mapsto a\star b$ permutes the elements of each line containing $a$.
We also define a \emph{flag} of $S$ to be an ordered pair $(\ell,p)$ where $\ell$ is a line of $S$ and $p$ is a point on $\ell$, and let $F$ denote the set of all flags of $S$.
The number of flags is $3|S|=\binom{t}{2}$. For each line $\ell$, we choose an arbitrary ordering of the points on $\ell$ and, for any $i\in[3]$, we let $\ell^{(i)}$ denote the $i^{th}$ point of $\ell$.

Given $n$, let $k$ be the smallest natural number such that ${t\choose 2}^k+3k\geq n$. 
Clearly, $k \le \frac{1}{\log{t\choose 2}}\log n + 1$. Let $G$ be the complete bipartite graph with parts $V\subseteq F^k$ and $K=[k]\times [3]$, with $|V|=n-3k$. We define a $t$-edge-colouring $c$ of $G$ (using the points of our Steiner system as colours) as follows: for each $f\in V$ and $(i,j)\in K$, let $c\left(\left\{f,(i,j)\right\}\right)=p\star\ell^{(j)}$, where $(\ell,p)$ is the $i^{th}$ component of $f$. To show that adding an edge between two vertices in $V$ creates a rainbow triangle, it suffices to show that every pair of such vertices is joined by either two disjoint rainbow paths of length two using disjoint sets of colours or three such paths that each use a different pair of colours from a set of three. Suppose $f$ and $f'$ are $k$-tuples of flags that differ in the $i^{th}$ component, say $f_i=(\ell,p)$ and $f'_i=(\ell',p')$. First, consider the possibility that $\ell=\ell'$ and $p\neq p'$. In this case, for every $j\in[3]$, $p\star\ell^{(j)}\neq p'\star\ell^{(j)}$, and neither is equal to $(p\star p')\star\ell^{(j)}$. Thus each path $f$--$(i,j)$--$f'$ is a rainbow path of length two using a distinct pair of colours from $\ell$. Next, if $\ell\neq\ell'$, then each edge $\left\{f,(i,j)\right\}$ is coloured with a different point from $\ell$ and each edge ${f',(i,j)}$ is coloured with a different point from $\ell'$ for $j\in[3]$. Since $\ell$ and $\ell'$ have at most one point in common, at most one path $f$--$(i,j)$--$f'$ is monochromatic. If this is the case, then the other two such paths are rainbow with disjoint sets of colours. Otherwise, all such paths are rainbow, and at most one pair of them have a colour in common, so there is a pair that uses disjoint sets of colours.

It is possible that adding an edge between two vertices in $K$ in some colour does not create a rainbow triangle; there are at most $\binom{|K|}{2}$ such edges. 
We can add these coloured edges to $(G,c)$ to form an $\mathfrak{R}(K_3)$-saturated $t$-edge-coloured graph $(G',c')$ with at most
\[
    |V| |K| + \binom{|K|}{2} \le (n-|K|)|K| + |K|^2 = n|K| \le \frac{3}{\log {t\choose 2}}n\log n + 3n
\]edges.
\end{proof}

When $t=3$, the coefficient of the $n \log n$ term in the upper bound is $\frac{3}{\log 3}$, while for large values of $t$ it is approximately $\frac{1.5}{\log t}$.

\section{Upper bounds for disconnected graphs}\label{sec:upperboundsdiscon}

In this section, we shall show that the rainbow saturation number of a disconnected graph can be bounded above by the rainbow saturation number of one of its connected components, up to additive $O(n)$ term.
Moreover, we shall show that if $H$ is a disconnected graph with no isolated vertices, then the $t$-rainbow saturation number of $H$ is at most $O(n\log{n})$ answering a question from \cite{sat} for disconnected graphs.
Throughout the section, we assume, for simplicity of exposition, that $H$ has no isolated vertices.

For a sequence of graphs $H_{1}, \dots, H_{k}$ we say that $H_{i}$ is \emph{maximal}, for some $i \in [k]$, if $H_{i}$ is not isomorphic to any proper subgraph of $H_{j}$ for any $j \in [k]$.
Observe that every sequence has a maximal element; for example, we can take one with the largest total number of vertices and edges.
\begin{proposition} \label{prop:disconnected}
    Let $H$ be a graph with connected components $H_{1}, \dots, H_{k}$ and let $H_{i}$ be a maximal component.
    Then, for every $t \ge e(H)$, we have
    \[
        \sat{t}{n}{H} \le \sat{t}{n}{H_{i}} + O(n).
    \]
\end{proposition}
\begin{proof}
    Without loss of generality we may assume that $i = 1$ and $H_1 \cong H_2 \cong \ldots \cong H_\ell$ (for some $\ell\in [k]$), and that no other component is isomorphic to $H_1$.  
    Let $H'=H_{\ell+1}\cup H_{\ell+2}\cup\ldots\cup H_{k}$. 

    Let $t'=e(H)\le t$ and consider the following graph $G$ on $n$ vertices. 
    First add vertex disjoint copies of all possible rainbow copies of $H'$ for every subset of size $|e(H')|$ in $[t']$. Write $V_1$ for the set of vertices spanned by these copies. 
    Second, consider the following coloured graph $H_1^{\star}$: for every set $A$ of colours of size $e(H_1)$ inside $[t']$, we add a rainbow of copy of $H_1$ with colours in $A$, where all rainbow copies share exactly one vertex. 
    Now we add $\ell-1$ vertex disjoint copies of $H_{1}^{\star}$ to $G$ and define $V_{2}$ to be the set of vertices spanned by these copies.
    In the set $V(G) \setminus (V_1 \cup V_2)$, consisting of the remaining vertices, we add a $\rainbow{H_{1}}$-saturated graph on $t$ colours. 
    It is easy to check that every non-edge in $V(G) \setminus (V_{1} \cup V_{2})$ is $\rainbow{H}$-saturated.
    Finally, if there are any non-edges which are not $\rainbow{H}$-saturated, we add those edges to $G$ in some colour that does not create a rainbow $H$. Clearly, there are at most $O(n)$ such edges.
    
    Let us show $G$ does not contain a rainbow copy of $H$. 
    Suppose for contradiction that it does.
    We shall obtain a contradiction by showing that the number of vertex disjoint rainbow copies of $H_{1}$ in $G$ is strictly smaller $\ell$.
    Note that $H_{1}$ cannot be a subgraph of $G[V_{1}]$ as, by construction, $H_{1}$ is not isomorphic to any connected component of $G[V_{1}]$ and, by maximality, $H_{1}$ cannot be a subgraph of any connected component of $G[V_{1}]$.
    Observe as well that each copy of $H_1^{\star}$ contains at most one rainbow copy of $H_1$.
    Finally, by construction, $V(G)\setminus V_1\cup V_2$ does not contain a rainbow copy of $H_1$.
    Therefore there are at most $\ell-1$ vertex disjoint rainbow copies of $H_{1}$.

    Let $p = |V_{1} \cup V_{2}|$.
    Observe that $p = \Theta(1)$ as $n$ goes to infinity.
    Therefore the number of edges in $G$ is at most $\binom{p}{2} + p(n-p) + \sat{t}{n-p}{H_{1}} \le pn + \sat{t}{n}{H_{1}} = \sat{t}{n}{H_{1}} + O(n)$.
\end{proof}

We have the following immediate corollary.
\begin{corollary}
    Let $H$ be a graph containing at least one component which is not a star and let $H'$ be a maximal component among the components of $H$ which are not stars.
    Then, for every $t \ge e(H)$, we have
    \[
        \sat{t}{n}{H} \le \sat{t}{n}{H'} + O(n) \le O(n\log{n}).
    \]
\end{corollary}

\begin{proof}
    Observe that $H'$ cannot be a subgraph of a star, hence by Proposition~\ref{prop:disconnected} and Corollary~\ref{corol:generalconnected}, we have that
    \[
        \sat{t}{n}{H} \le \sat{t}{n}{H'} + O(n) \le O(n\log{n}).
    \] \end{proof}

We showed that if a disconnected graph contains a component which is not a star then its rainbow saturation number is subquadratic.
Since stars have rainbow saturation number which is quadratic in $n$, one might suspect that the same should hold for disconnected graphs where each component is a star.
The following proposition shows that this is not the case.

\begin{proposition} \label{prop:stars}
    Let $H = S_{1} \cup S_{2} \cup \cdots \cup S_{k}$ be a graph with more than one component, each of which is a star.
    Then for every $t \ge e(H)$ we have 
    \[
        \sat{t}{n}{H} \le O(n).
    \]
\end{proposition}
\begin{proof}
    Suppose $|S_{1}| \le |S_{2}| \le \cdots \le |S_{k}|$.
    First we shall show the case when $k=2$.
    Let $a = |S_{1}| - 1$ and $b = |S_{2}| - 1$.
    Let $G = K \cup L$ where $G[K]$ is a complete graph of size $a + b - 1$ and $L$ is an independent set of size $n - |K|$.
    Let $K = \left\{ x_{1}, \dots, x_{a+b-1} \right\}$.
    First we join every vertex $x_{i} \in K$ with every vertex $y \in L$ and give the edge colour $i$.
    Next we shall describe the colouring of the edges inside $K$.
    Let $x_{i},x_{j} \in K$ where $i \le j$.
    If $i \le a$ and $j \ge a$ then assign $a+b$ as the colour of $x_{i}x_{j}$, otherwise assign $j$ as the colour of $x_{i}x_{j}$.

    We claim that there is no rainbow copy of $S_{1} \cup S_{2}$ in $G$.
    To see that, observe first that every rainbow copy of $S_{i}$ in $G$ uses at least $|S_{i}| - 1$ vertices of $K$.
    Indeed, suppose for contradiction that it is not the case and that there is a rainbow copy of $S_{i}$ which uses fewer than $|S_{i}| - 1$ vertices of $K$.
    Then it must use at least two vertices, say $x, y$, of $L$.
    It follows from independence of $L$ that the center $z$ of that rainbow copy must be in $K$.
    We obtain a contradiction by noticing that $zx$ and $zy$ have the same colour.
    Therefore if there is a rainbow copy of $S_{1} \cup S_{2}$ then it has to use at least $a+b$ vertices of $K$, which is a contradiction since there are only $a+b-1$ such vertices.

    Next we shall show that every non-edge is $\rainbow{H}$-saturated.
    Consider any non-edge $xy$ in $L$ and any colour $c \in [t]$.

    If $c \le a$ then we find a copy of $S_{1}$ in $\left\{ x, y, x_{1}, \dots, x_{a} \right\} \setminus \left\{ x_{c} \right\}$ with $x$ being the center and a copy of $S_{2}$ in $\left\{ x_{c}, x_{a+1}, \dots, x_{a+b-1}, z \right\}$ with $x_{a+1}$ as the center, for any $z \in L \setminus \left\{ x, y \right\}$.
    Observe that those two copies are vertex disjoint and the copy of $S_{1}$ uses only colours from $[a]$ and the copy of $S_{2}$ uses colours from $[a+1, a+b]$.
    Hence we have a rainbow copy of $S_{1} \cup S_{2}$.

    If $c \in [a+1, a+b-1]$ then we find a copy of $S_{2}$ in $\left\{ x, y, x_{a}, \dots, x_{a+b-1} \right\} \setminus \left\{ x_{c} \right\}$ with $x$ being the center and a copy of $S_{1}$ in $\left\{x_{1}, \cdots, x_{a-1}, x_{c}, z \right\}$ with $x_{1}$ as the center, for any $z \in L \setminus \left\{ x, y \right\}$.
    Observe that those two copies are vertex disjoint and the copy of $S_{1}$ uses only colours from $[a-1]\cup\left\{ a+b \right\}$ and the copy of $S_{2}$ uses colours from $[a, a+b-1]$.
    Hence we have a rainbow copy of $S_{1} \cup S_{2}$.

    In the remaining case when $c \ge a+b$, it is easy to check that we can find a rainbow copy of $S_{1} \cup S_{2}$ where both of the centers are in $L$.

    Observe that we have $e(G) \le (n - |K|)|K| + |K|^{2} = |K|n = (a+b-1)n = (|S_{1}| + |S_{2}|-3)n$.

    Now, suppose $k \ge 3$. 
    We let $t^{\star} = e(H)$. Moreover, let $G = G' \cup G''$ where $G'$ is an $\rainbow{S_{1} \cup S_{2}}$-saturated graph on $n'=n - (k-2)(t^{\star}+1)$ vertices with $\sat{t^{\star}}{n'}{S_1\cup S_2}$ edges and $G''$ is the vertex-disjoint union of $k-2$ rainbow copies of $t^{\star}$-stars.
    It is easy to check that there is no rainbow copy of $H$ in $G$.
    Indeed, by assumption there can not be two vertex-disjoint rainbow copies of distinct components of $H$ appearing in $G'$.
    Note as well that there can only be at most $k-2$ vertex-disjoint stars in $G''$, hence in total there are at most $k-1$ disjoint rainbow components of $H$ in $G$.
    Finally, it is clear that the addition of any coloured non-edge inside $G'$ creates a rainbow copy of $H$. 
    Now, we keep adding edges to $G$ (with both endpoints in $G''$ or with one endpoint in $G'$ and one in $G''$) until $G$ is saturated.
    The case $k=2$ shows that $\sat{t^{\star}}{n}{(S_1\cup S_2)} \le (|S_{1}| + |S_{2}|-3)n$, hence, the number of edges in $G$ is at most $(n-|G''|)|G''| + |G''|^{2} + e(G') \le n|G''| + (|S_{1}|+|S_{2}|-3)n \le O(n)$.
\end{proof}

We have the following corollary from Propositions~\ref{prop:disconnected}~and~\ref{prop:stars}.
\begin{corollary}
    Let $H$ be a disconnected graph.
    Then for every $t \ge e(H)$ we have
    \[
        \sat{t}{n}{H} \le O(n \log{n}) = o(n^{2}).
    \]
\end{corollary}

\section{Deducing the main results} \label{sec:proofs}
We are now ready to deduce Theorems~\ref{thm:class}~and~\ref{thm:largenumbercolours}. 

\begin{proof}[\textbf{Proof of Theorem}~\ref{thm:class}]
    First, note that item~\ref{thm:class:star} is a result appearing in \cite{sat} and item~\ref{thm:class:Hkl} is just a restatement of Theorem~\ref{thm:rotatededge}. 
    Now, the lower bounds in items~\ref{thm:class:conical}, \ref{thm:class:triangular} follow by Theorems~\ref{thm:lower_bound_triangles}~and~\ref{thm:max_degree}, respectively, and the upper is a consequence of Theorem~\ref{thm:cycles} since in both cases $H$ must contain a cycle. 

    In item~\ref{thm:class:linear} the lower bound follows from Lemma~\ref{lem:lower_bound_linear} and the upper bound follows from Theorem~\ref{thm:trees_2_conn}. 
\end{proof}

\begin{proof}[\textbf{Proof of Theorem}~\ref{thm:largenumbercolours}]
    Observe first that if $H$ is a connected graph on at most four vertices which contains a leaf and no conical vertex, then $H$ must be a path on four vertices, hence by Theorem~\ref{thm:class}\ref{thm:class:linear} its $t$-rainbow saturation number is linear.
    We may therefore assume that $|H|\ge 5$. 
    Let $xy$ be a pendant edge of $H$. 
    If $H\setminus \{x,y\}$ is not a clique then we are done by Theorem~\ref{thm:bfvw}.
    Hence, we may then assume $H=H_{k, \ell}$ for some $k\ge 3$ and $\ell \le k-1$. 
    Suppose $\ell \le k-2$, then result follows by Theorem~\ref{thm:Hkl}. 
    Hence, we may assume $\ell=k-1$ in which case $k$ must be odd, by assumption, and therefore $H$ is a $K_{k+1}$ with a rotated edge, so we are done by Theorem~\ref{thm:rotatededge}. 
\end{proof}

\section{Concluding remarks}\label{sec:openproblems}
We have shown that for any $t\geq {r \choose 2}$, $\sat{t}{n}{K_r}=\Theta(n\log n)$ when $n\to\infty$, i.e., there exist constants $c_1=c_1(t,r)$ and $c_2=c_2(t,r)$ such that $c_1 n\log n\leq \sat{t}{n}{K_r}\leq c_2 n\log n$. 
There is still an enormous gap between our lower and upper bounds.
In an earlier version of this manuscript, we conjectured that the true value is closer to the upper bound:

\begin{conjecture} \label{conjecture:complete_our}
For every $r\geq 3$, there exists a constant $c(r)>0$ such that for every $t\geq{r\choose 2}$, \[\sat{t}{n}{K_r}\geq \frac{c(r)}{\log t}n\log n\] for all $n\geq n_0(t)$.
\end{conjecture}

Recently, Kor\'andi~\cite{Korandisat} resolved Conjecture~\ref{conjecture:complete_our} by showing the following.

\begin{theorem}[\cite{Korandisat}]
For every $r\geq 3$, and any $t\geq{r\choose 2}$, \[\sat{t}{n}{K_r}\geq \frac{t(1+o(1))}{(t-r+2)\log(t-r+2)}n\log n\] as $n\rightarrow \infty$, with equality for $r=3$. 
\end{theorem}

This theorem, together with Theorem~\ref{thm:randomcomplete} gives: \[\sat{t}{n}{K_r}=\Theta_{r}\Big (\frac{n\log n}{\log t} \Big )\]. 

Now, when $H$ is an even clique with a rotated edge, we know that $\sat{t}{n}{H}$ is always $\Theta(n)$ for $t\geq e(H)$. However, for odd cliques with rotated edges, we do not even know the asymptotic behaviour of $\sat{t}{n}{H}$ for large values of $t$.

\begin{question} \label{question:rotated}
If $H$ is a copy of $K_r$ with a rotated edge (as shown in Figure~\ref{fig1}) for some \emph{odd} $r\geq 5$ and $t\geq{r\choose 2}$, what is the asymptotic growth rate of $\sat{t}{n}{H}$?
\end{question}

\begin{figure}[htbp]
\centering
\begin{tikzpicture}
[
  every node/.style={draw,circle,inner sep=2pt},
  fsnode/.style={fill=black},
  ssnode/.style={fill=black},
  every fit/.style={ellipse,draw,inner sep=-2pt,text width=2mm},
  shorten >= 2pt,shorten <= 2pt, auto, thick]

\foreach \a in {1,2,...,5}{
\node [fsnode, anchor=center] at ({360/5 * (\a-1)}:1.5 cm+ 1)(f\a) {};
}
\begin{scope}[xshift=33.75mm,yshift=0cm,start chain=going right,node distance=10mm]
  \node[ssnode,on chain] (s1) {};
\end{scope}

\draw[black](f1) -- (f2) [dashed];
\coordinate (a) at ($(f1)!0.5!(f2)$);
\draw[black](f1) -- (s1);
\coordinate (b) at ($(f1)!0.5!(s1)$);
\draw[->, bend right=90] (a) to [out=60, in=120] (b);
\draw[black](f1) -- (f3);
\draw[black](f1) -- (f4);
\draw[black](f1) -- (f5);
\draw[black](f2) -- (f3);
\draw[black](f2) -- (f4);
\draw[black](f2) -- (f5);
\draw[black](f3) -- (f4);
\draw[black](f3) -- (f5);
\draw[black](f4) -- (f5);
\end{tikzpicture}
\caption{$K_5$ with a rotated edge. The dashed line represents the removed edge.}
\label{fig1}
\end{figure}
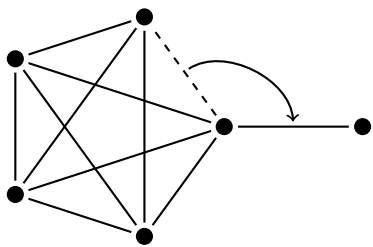

    The following conjecture together with Theorem~\ref{thm:class} and Question~\ref{question:rotated} would completely classify the possible rates of growth of $\sat{t}{n}{H}$ for all connected graphs $H$ and every constant $t\geq e(H)$.

\begin{conjecture}
    Let $H$ be a connected graph (other than an odd clique with a rotated edge) with an edge not in a triangle and no conical vertex. 
    Then, for every $t\geq e(H)$, $\sat{t}{n}{H}=O(n)$.
\end{conjecture}

    Note that we can confirm this conjecture when the number of available colours is at least ${|H|-1 \choose 2}$. 
    Indeed, either $H$ is in one of the classes defined in Theorem~\ref{thm:class}, in which case we are done, or $H$ has a leaf and is not a clique with a rotated edge, hence by Theorem~\ref{thm:largenumbercolours} we have $\sat{t}{n}{H} = \Theta(n)$.

One different direction would be to allow the palette of colours to be infinite. 
We have only considered this question for complete graphs and showed that $\sat{}{n}{K_{r}} \le 2(r-2)n$ for any $r \ge 3$.

Recall that the construction in Corollary~\ref{cor:randomcomplete} is a disjoint union of an independent set $A$ and two $(r-2)$-cliques $B$ and $C$, such that all the edges between $A$ and $B \cup C$ are present and all the edges in $B$, $C$ and between $A$ and $B\cup C$ receive different colours.
We conjecture that, for $n \ge 2(r-2)$, the above construction is best possible up to the configuration of the edges between $B$ and $C$.
\begin{conjecture}
For any integer $r \ge 3$, there exists a constant $C_{r}$ depending only on $r$ such that, for any $n \ge 2(r-2)$,
\[
    \sat{}{n}{K_r}= 2(r-2)n+C_{r}.
\]
\end{conjecture}

Finally, we conjecture that, like the ordinary saturation numbers, the rainbow saturation numbers of any graph are at most linear in $n$.
\begin{conjecture}
For any graph $H$, $\sat{}{n}{H} = O(n)$.
\end{conjecture}

\bibliography{rainbow}{}
\bibliographystyle{plain}

\end{document}